\documentclass[oneside,a4paper,reqno]{amsart}
\usepackage{pdfsync, enumitem}
\usepackage{amsmath,bm}
\usepackage{mathtools}
\usepackage{graphicx}
\usepackage[all,cmtip]{xy}
\usepackage{tikz-cd}
\usepackage{fullpage}
\usepackage[margin=2.5cm, top=3cm, bottom=3cm, footskip=1cm]{geometry}

\usepackage{stmaryrd}
\usepackage{mathrsfs}
\usepackage{hyperref}
\usepackage{tikz-cd}
\usepackage{amssymb}
\usepackage{stackrel}
\usepackage{adjustbox}
\usepackage{multicol}
\usepackage{amsmath, amsthm, amscd, amssymb, latexsym, eucal}
\usepackage[all]{xy}
\def\serieslogo@{} \def\@setcopyright{} \makeatother

\usepackage[colorinlistoftodos]{todonotes}

\usepackage{hyperref}
\usepackage{color}
\usepackage{cite}
\usepackage{quiver}
\usepackage{marvosym}

\makeatletter
\renewcommand*\env@matrix[1][c]{\hskip -\arraycolsep
	\let\@ifnextchar\new@ifnextchar
	\array{*\c@MaxMatrixCols #1}}
\makeatother

\usepackage{color}

\numberwithin{equation}{section}
\newtheorem{thm}{Theorem}[section]
\newtheorem*{main-thm}{Theorem}
\newtheorem*{Auslander-thm}{Auslander's Theorem}

\newtheorem{cor}[thm]{Corollary}
\newtheorem{lem}[thm]{Lemma}
\newtheorem{prop}[thm]{Proposition}

\theoremstyle{definition}
\newtheorem{defn}[thm]{Definition}
\newtheorem{rem}[thm]{Remark}
\newtheorem{recollection}[thm]{Recollection}
\newtheorem{exmp}[thm]{Example}

\newtheorem*{ackn}{Acknowledgement}

\newtheorem*{example}{Example}
\newtheorem*{connot}{Notation and Conventions}

\newcommand\padova{%
\mathrel{\ooalign{\hss{\scalebox{0.5}{$\longleftrightarrow$}}\hss\cr%
\kern0.9ex\raise0.55ex\hbox{\scalebox{0.5}{$\boldsymbol{\bm{\vert}}$}}}}}
\newcommand\hfpadova{%
\mathrel{\ooalign{\hss{\scalebox{0.5}{$\longleftrightarrow$}}\hss\cr%
\kern0.9ex\raise0.55ex\hbox{\scalebox{0.5}{$\boldsymbol{\mathsf{h}\mathsf{f}\ \ }$}}}}}
\newcommand\rpadova{%
\mathrel{\ooalign{\hss{\scalebox{0.5}{$\longrightarrow$}}\hss\cr%
\kern0.9ex\raise0.55ex\hbox{\scalebox{0.5}{$\boldsymbol{\bm{\vert}}$}}}}}
\newcommand\lpadova{%
\mathrel{\ooalign{\hss{\scalebox{0.5}{$\longleftarrow$}}\hss\cr%
\kern0.9ex\raise0.55ex\hbox{\scalebox{0.5}{$\boldsymbol{\bm{\vert}}$}}}}}

\newcommand{\T}{\mathcal T}

\newcommand{\X}{\mathcal X}

\DeclareMathOperator{\pd}{\mathsf{pdim}}

\DeclareMathOperator*{\gd}{\mathsf{gl.dim}}

\DeclareMathOperator*{\Mod}{\mathsf{Mod}-\!}

\DeclareMathOperator*{\smod}{\mathsf{mod}-\!}

\DeclareMathOperator*{\Inj}{\mathsf{Inj}-\!}

\DeclareMathOperator*{\Proj}{\mathsf{Proj}-\!}

\DeclareMathOperator*{\Add}{\mathsf{Add}}

\DeclareMathOperator*{\add}{\mathsf{add}}

\DeclareMathOperator{\Hom}{\mathsf{Hom}}

\usepackage{stackengine}

\newsavebox{\proofbox}
\savebox{\proofbox}{\begin{picture}(7,7)%
	\put(0,0){\framebox(7,7){}}\end{picture}}

\usepackage{latexsym}
\usepackage{pstricks}
\usepackage{comment}

\usepackage[greek,english]{babel}

\begin{document}

\title{Global dimension of dg algebras via compact silting objects}

\author[Kostas]{Panagiotis Kostas}
\address{Department of Mathematics, Aristotle University of Thessaloniki, Thessaloniki 54124, Greece}
\email{pkostasg@math.auth.gr}


\begin{abstract} 
    We introduce a notion of global dimension for a triangulated category relative to a compact silting object. We prove that the finiteness of this dimension is an intrinsic property of the triangulated category itself and, therefore, independent of the choice of the silting object. Focusing on the setup of connective differential graded (dg) algebras, we analyse the behaviour of global dimension under dg algebra homomorphisms and establish explicit bounds. This allows us to deduce a bound for the global dimension of certain dg quiver algebras. We also relate the regularity of the big singularity category of a proper connective dg algebra to the finiteness of its global dimension. 
\end{abstract}

\maketitle
\setcounter{tocdepth}{1}

\section{Introduction} The global dimension of a ring is an important invariant in representation theory. There have been various approaches to understanding global dimension in the context of triangulated and dg categories, most importantly through strong generation \cite{bondal van den bergh} and smoothness \cite{kontsevich}, respectively. The former is the topic of important recent work \cite{neeman} in algebraic geometry. It is known that the global dimension of rings is not preserved under derived equivalences (see \cite{fang_hu_koenig}), but its finiteness is, i.e.\ given a triangle equivalence $\mathsf{D}(R)\simeq \mathsf{D}(S)$ of derived categories of rings $R$ and $S$, then $\gd R<\infty$ if and only if $\gd S<\infty$. This can be understood via a very general construction from \cite{regular}; for a compactly generated triangulated category $\mathcal{T}$, we may consider the ``right far-away orthogonal'' of the compact objects $\mathcal{T}^{\mathsf{c}}$, in order to define $(\mathcal{T}^{\mathsf{c}})^{\padova}\eqqcolon \mathcal{T}^{\mathsf{b}}$, the \emph{bounded objects} in $\mathcal{T}$. Further, the ``left far-away orthogonal'' ${^{\padova}}(\mathcal{T}^{\mathsf{b}})\eqqcolon \mathcal{T}^{\mathsf{b}}_p$ defines the \emph{bounded projective objects} in $\mathcal{T}$, see Definition \ref{intrinsic subcategories} and \cite{regular} for details. The identification of the latter (intrinsic) subcategories is interpreted as finite global dimension for $\mathcal{T}$; in the case of $\mathsf{D}(R)$, the derived category of a ring, it means precisely that $\mathsf{D}^{\mathsf{b}}(\Mod R)=\mathsf{K}^{\mathsf{b}}(\Proj R)$, i.e.\ that $\gd R$ is finite. Since the global dimension of rings is not a derived invariant, this approach cannot be used in order to assign a specific number that measures the finiteness of global dimension for a triangulated category; some non-intrinsic choices have to be made. 

Our first aim in this paper is to define the global dimension of a compactly generated triangulated category $\mathcal{T}$ relative to a compact silting object (when it exists). The main result of this paper, proved in Theorem~\ref{global dim vs relative global dim}, reveals that the finiteness of this global dimension is equivalent to the inclusion $\mathcal{T}^{\mathsf{b}}\subseteq \mathcal{T}^{\mathsf{b}}_p$ and it is, therefore, independent of the chosen compact silting object. For a connective dg algebra $A$, the global dimension of its derived category $\mathsf{D}(A)$ relative to $A$ is commonly used \cite{minamoto,tomonaga} as the global dimension of $A$ (see also Remark \ref{remark} and Remark \ref{gldim vs minamoto gldim}). It follows, as a consequence of the theorem mentioned above, that the finiteness of global dimension of $A$ is a property intrinsic to $\mathsf{D}(A)$. We accompany this section with several other criteria for detecting finite global dimension (some of which are necessary in order to compare with \cite{minamoto,tomonaga}), see (\ref{implications}) for a summary.

In \cite[Proposition 5.9(iii)]{regular} it is shown that $\mathsf{K}_{\mathsf{ac}}(\Inj\Lambda)$, where $\Lambda$ is an Artin algebra, is regular (in the sense of \cite[Definition 5.7]{regular}) if and only if $\gd\Lambda<\infty$. In Section \ref{singularity} we prove the analogous result for proper connective dg algebras, see Corollary \ref{regularity for big singularity}.

Investigating the behaviour of the global dimension of a finite dimensional algebra in terms of its Gabriel quiver (and the relations) has been an active field of research, see the survey \cite{happel_zacharia}. In the dg world, much less seems to be known in this direction. We list some of these instances below.  
\begin{itemize}
    \item[$\bullet$] Smoothness for certain graded gentle algebras (and with trivial differential) was investigated in \cite[Subsection 3.4]{haiden_katzarkov_kontsevich} and \cite[Subsection 3.1]{lekili_polishchuk}. 
    \item[$\bullet$] Explicit smooth algebras were constructed in \cite[Subsection 4.3]{orlov}, using twisted tensor products. 
    \item[$\bullet$] The global dimension of a path algebra that is given a connective differential grading was investigated in \cite[Subsection 3.2]{tomonaga}. 
\end{itemize}

Motivated by the third item above, we study the global dimension of a bound quiver algebra $kQ/I$ that is given a connective dg algebra structure and for which $Q$ does not have oriented cycles. We prove that the latter is finite and bounded explicitly in terms of the maximal length of a path in the quiver, see Proposition \ref{acyclic quiver}. This is a consequence of our second main theorem, which is about homomorphisms of dg algebras and global dimension, inspired by results of Efimov--Orlov \cite{efimov_orlov}, see Theorem \ref{dg algebra homomorphism}.

\subsection*{Relation to Other Work} The notion of global dimension for triangulated categories relative to t-structures is not new. The definition we propose (Definition \ref{main definition}) is given having in mind the definition of global dimension for connective dg algebras after Minamoto \cite{minamoto}. A notion of projective dimension, in the context of a compactly generated triangulated category with a single compact generator, is given in \cite[Appendix B]{biswas_chen_manali-rahul_parker_zheng} and coincides with the one we propose (for the t-structure induced by the generator). This was used in \cite{gldim for approximable} to introduce a notion of global dimension for approximable triangulated categories and show that its finiteness (under a boundedness condition on the strong generator) is an intrinsic property of the triangulated category studied, see Remark \ref{relation to other work} for more details.

\begin{ackn}
    I thank an anonymous referee for valuable suggestions and comments and in particular for suggesting Remark \ref{koszul} and Corollary \ref{not reflexive}.
\end{ackn}

\begin{connot}
Given a subcategory $\mathcal{X}$ in a triangulated category $\T$ and an interval $I$ of integers, we consider the subcategories
$$\mathcal{X}^{\perp_{I}}\coloneqq\{y\in\T\colon \Hom_\T(\mathcal{X},y[i])=0,\forall i \in I\}$$
where $I$ is, sometimes, represented by symbols $>0$ or $\leq 0$ with the obvious meaning. The symbol $\perp$ stands simply for $\perp_0$. We denote by $\mathsf{Add}(\X)$ (resp. $\mathsf{add}(\X)$) the subcategory formed by the summands of the existing coproducts (resp. finite coproducts) in $\mathcal{T}$ of objects in $\X$. We use the term dg algebra for a differential graded algebra over -- unless otherwise stated -- a commutative ring $k$. By a module over a dg algebra we mean a right dg module. 
\end{connot}

\section{Global dimension for connective dg algebras} \label{section 2}

We recall that a \emph{t-structure} on a triangulated category $\mathcal{T}$ is a torsion pair $(\mathcal{U},\mathcal{V})$ on $\mathcal{T}$ (i.e.\ a pair of subcategories satisfying $\mathcal{U}={^{\perp}\mathcal{V}}$, $\mathcal{U}^{\perp}=\mathcal{V}$ and $\mathcal{T}=\mathcal{U}\ast\mathcal{V}$) such that $\mathcal{U}[1]\subseteq \mathcal{U}$. The \emph{heart} of this t-structure is defined as the intersection $\mathcal{H}\coloneqq \mathcal{U}\cap \mathcal{V}[1]$ (and it is an abelian category \cite{BBD}).
 
\begin{exmp} (\emph{Silting objects})
    An object $M$ of a (small) triangulated category $\mathcal{D}$ is called a \emph{silting generator} if $\mathsf{Hom}_{\mathcal{D}}(M,M[>\!0])=0$ and $\mathsf{thick}(M)=\mathcal{D}$. We say that an object $M$ of a compactly generated triangulated category $\mathcal{T}$ is a \emph{compact silting object} if it is a silting generator of $\mathcal{T}^{\mathsf{c}}$. We then know from \cite[Theorem 2.1]{hoshino_kato_miyachi} that the pair $(M^{\perp_{>0}},M^{\perp_{\leq 0 }})$ is a t-structure of $\mathcal{T}$ whose heart is $M^{\perp_{\neq 0}}\eqqcolon \mathcal{H}_{M}$.
\end{exmp}

We introduce a notion of projective dimension of objects in a triangulated category, relative to a heart of a given t-structure. Similar considerations have appeared in \cite[Appendix B]{biswas_chen_manali-rahul_parker_zheng} and in \cite{tomonaga}.

\begin{defn} \label{main definition}
    Let $\mathcal{T}$ be a triangulated category and $(\mathcal{U},\mathcal{V})$ a t-structure in $\mathcal{T}$ with heart $\mathcal{H}$.
    \begin{itemize}
        \item[$\bullet$] The \emph{projective dimension relative to} $\mathcal{H}$ of an object $x$ of $\mathcal{T}$, denoted by $\pd_{\mathcal{H}}x$, is defined to be the infimum $n\in\mathbb{Z}\cup\{-\infty,+\infty\}$ for which 
        \[
         \mathsf{Hom}_{\mathcal{T}}(x,y[>\!n])=0
         \]
         for all $y\in\mathcal{H}$. 
         \item[$\bullet$] The \emph{global dimension of $\mathcal{T}$ relative to $\mathcal{H}$}, denoted by $\gd{_{\mathcal{H}}}\mathcal{T}$, is defined to be the supremum among  $\pd_{\mathcal{H}}x$ for every object $x$ in the heart $\mathcal{H}$. 
    \end{itemize}
\end{defn}

The projective dimension of an object can in principle be negative, but it is always nonnegative for objects in the aisle, since $\mathsf{Hom}_{\mathcal{T}}(\mathcal{U},\mathcal{H}[<0])=0$. For $\mathcal{T}=\mathsf{D}(R)$, the derived category of a ring $R$, we have $\gd_{\mathcal{H}}\mathcal{T}=\gd R$, where $\mathcal{H}\simeq \Mod R$ is the heart of the canonical t-structure in $\mathsf{D}(R)$.

\begin{recollection} \label{recollection} 
    For the theory of dg categories and their derived categories we refer the reader to \cite{keller deriving}. By a dg algebra we mean a dg category with a single object. 
    Let now $A$ be a connective dg algebra over a commutative ring $k$ (i.e.\ $H^n(A)=0$ for all $n>0$). Then $A$ is a silting object in $\mathcal{T}=\mathsf{D}(A)$ and therefore it induces a t-structure $(A^{\perp_{>0}},A^{\perp_{\leq 0}})$, where explicitly 
    \[
    A^{\perp_{>0}}=\{x\in\mathsf{D}(A): H^n(x)=0 \text{ for }n>0\}\eqqcolon \mathsf{D}^{\leq 0}(A)\]
    and
    \[
    A^{\perp_{\leq 0}}=\{x\in\mathsf{D}(A): H^n(x)=0 \text{ for }n\leq 0\}\eqqcolon \mathsf{D}^{>0}(A).
    \]
    There is an equivalence $\mathcal{H}\simeq \Mod H^0(A)$, where $\mathcal{H}$ is the heart of the t-structure above; explicitly we may assume that $A^n=0$ for $n>0$ and then, there is a canonical morphism of dg algebras $A\rightarrow H^0(A)$ which induces a functor $\mathsf{D}(H^0(A))\rightarrow \mathsf{D}(A)$. This functor restricts to an equivalence $\Mod H^0(A)\rightarrow\mathcal{H}$; see \cite[Section 2.1]{amiot} for details. 
\end{recollection}
 
The following example of relative global dimension is the main object of study in this paper. 

\begin{defn} \label{examples} 
The \emph{global dimension} of a connective dg algebra $A$ is defined to be the global dimension of $\mathsf{D}(A)$ relative to $\mathcal{H}\simeq \Mod H^0(A)$. For an object $x$ of $\mathsf{D}(A)$, we simply denote $\pd_{\mathcal{H}}x$ by $\pd_Ax$.
\end{defn}

\begin{rem}
    We will later observe, see Remark \ref{gldim vs minamoto gldim}, that the definition of global dimension for connective dg algebras, as in Definition \ref{examples}, agrees with the definition given in \cite[Theorem 5.1]{minamoto}. Note, however, that the two notions of projective dimension used here and in \cite{minamoto} differ; for a connective dg algebra $A$ and any integer $n$, we have $\pd_AA[n]=n$ while using the definition from \cite{minamoto, yekutieli duality}, the dg module $A[n]$ has projective dimension zero. 
\end{rem}

We now observe the following property of projective dimension. 

\begin{lem} \label{bound of projective dimensions in a triangle}
    Let $\mathcal{T}$ be a triangulated category and $(\mathcal{U},\mathcal{V})$ a t-structure in $\mathcal{T}$ with heart $\mathcal{H}$. Then for any triangle $x\rightarrow y\rightarrow z\rightarrow x[1]$ of $\mathcal{T}$, the inequality $\pd_{\mathcal{H}}y\leq \mathsf{max}\{\pd_{\mathcal{H}}x,\pd_{\mathcal{H}}z\}$ holds. 
\end{lem}
\begin{proof}
    This follows immediately by applying $\mathsf{Hom}_{\mathcal{T}}(-,h)$ to the given triangle, for any $h\in\mathcal{H}$. 
\end{proof}

We recall a few definitions from \cite{regular}; for a triangulated category $\mathcal{T}$ and any collection of objects $\mathcal{X}$ of $\mathcal{T}$, we define the following subcategories 
\[
{^{\perp_{\gg}}}{\mathcal{X}}\coloneqq\{t\in\mathcal{T}:\forall x\in\mathcal{X}, \  \mathsf{Hom}_{\mathcal{T}}(t,x[n])=0 \text{ for }n\gg 0\}
\]
\[
{^{\perp_{\ll}}}{\mathcal{X}}\coloneqq\{t\in\mathcal{T}:\forall x\in\mathcal{X}, \  \mathsf{Hom}_{\mathcal{T}}(t,x[n])=0 \text{ for }n\ll 0\}
\]
and furthermore we set ${^{\padova}}\mathcal{X}={^{\perp_{\gg}}}{\mathcal{X}}\cap {^{\perp_{\ll}}}{\mathcal{X}}$. Similarly we define such ``orthogonals" from the right; we only need one of them, namely
\[
\mathcal{X}^{\padova}\coloneqq\{t\in\mathcal{T}:\forall x\in\mathcal{X}, \  \mathsf{Hom}_{\mathcal{T}}(x,t[n])=0 \text{ for }|n|\gg 0\}.
\]
The above are thick triangulated subcategories of $\mathcal{T}$, see \cite[Lemma 2.9]{regular}. 

\begin{defn} (\!\!\cite{regular}) \label{intrinsic subcategories}
    For any compactly generated triangulated category $\mathcal{T}$, we consider the following triangulated subcategories:
    \begin{itemize}
        \item[$\bullet$]  the subcategory of \emph{bounded objects} $\mathcal{T}^{\mathsf{b}}\coloneqq (\mathcal{T}^{\mathsf{c}})^{\padova}$; 
        \item[$\bullet$] the subcategory of \emph{bounded projective objects} $\mathcal{T}^{\mathsf{b}}_p\coloneqq {^{\padova}}(\mathcal{T}^{\mathsf{b}})$;
        \item[$\bullet$] the subcategory of \emph{bounded injective objects} $\mathcal{T}^{\mathsf{b}}_i\coloneqq (\mathcal{T}^{\mathsf{b}})^{\padova}$. 
    \end{itemize}
    We say that $\mathcal{T}$ has \emph{finite global dimension} if $\mathcal{T}^{\mathsf{b}}=\mathcal{T}^{\mathsf{b}}_p$ (or equivalently $\mathcal{T}^{\mathsf{b}}=\mathcal{T}^{\mathsf{b}}_i$, see \cite[Lemma 4.1]{regular}). 
\end{defn}

\begin{example}
    For a dg algebra $A$ and $\mathcal{T}=\mathsf{D}(A)$, it holds that $\mathcal{T}^{\mathsf{b}}=\mathsf{D}^{\mathsf{b}}(A)$, i.e.\ consists of the modules with bounded cohomology. If $A$ is connective and has bounded cohomology, then $\mathcal{T}^{\mathsf{b}}_p=\mathsf{thick}(\mathsf{Add}A)$ and $\mathcal{T}^{\mathsf{b}}_i=\mathsf{thick}(\mathsf{Prod}A^*)$, see \cite[Corollary 3.16]{regular}.
\end{example}

We will now relate the two concepts of finite global dimension that have appeared when $\mathcal{T}$ has a compact silting object, for which we need the following description of the bounded objects in this case.

\begin{prop}  \label{bounded t structure} \textnormal{(\!\!\cite[Proposition 3.8]{regular})}
    Let $\mathcal{T}$ be a compactly generated triangulated category and $M$ a compact silting object. The t-structure $(M^{\perp_{>0}},M^{\perp_{\leq 0}})$ restricts to a t-structure in $\mathcal{T}^{\mathsf{b}}$ and it holds that
    \[
    \mathcal{T}^{\mathsf{b}}=(\cup_{s\in\mathbb{Z}}(M^{\perp_{>s}}\cap\mathcal{T}^{\mathsf{b}}))\cap (\cup_{t\in\mathbb{Z}}(M^{\perp_{\leq t}}\cap\mathcal{T}^{\mathsf{b}}))
    \]
    i.e.\ the restricted t-structure is a bounded t-structure of $\mathcal{T}^{\mathsf{b}}$. In particular, $\mathcal{T}^{\mathsf{b}}=\mathsf{thick}(\mathcal{H}_M)$.
\end{prop}

\begin{thm} \label{global dim vs relative global dim}
    Let $\mathcal{T}$ be a compactly generated triangulated category that admits a compact silting object $M$. The following are equivalent. 
    \begin{itemize}
        \item[\textnormal{(i)}] $\gd{_{\mathcal{H}_M}}\mathcal{T}<\infty$. 
        \item[\textnormal{(ii)}] $\mathcal{T}^{\mathsf{b}}\subseteq \mathcal{T}^{\mathsf{b}}_p$.
        \item[\textnormal{(iii)}] $\mathcal{T}^{\mathsf{b}}\subseteq \mathcal{T}^{\mathsf{b}}_i$.
    \end{itemize}
    In particular when $M\in M^{\padova}$, then $\gd{_{\mathcal{H}_M}}\mathcal{T}$ is finite if and only if $\mathcal{T}$ has finite global dimension. 
\end{thm}
\begin{proof}
      (i)$\implies$(ii): Assume that $\gd{_{\mathcal{H}_M}}\mathcal{T}<\infty$. Then, by the definition of the global dimension, it follows that $\mathcal{H}_M\subseteq {^{\perp_{\gg}}}\mathcal{H}_M$. Moreover, we have $\mathsf{Hom}_{\mathcal{T}}(M^{\perp_{>0}},\mathcal{H}_{M}[<0])=0$, since $\mathcal{H}_{M}[<0]$ is contained in the coaisle $M^{\perp_{\leq 0}}$. In particular, it follows that $\mathcal{H}_{M}\subseteq {^{\perp_{\ll}}}\mathcal{H}_M$ and therefore we can conclude from the above that $\mathcal{H}_M\subseteq {^{\padova}}\mathcal{H}_M$. By considering thick closures inside $\mathcal{T}$, we get an inclusion 
       \[
       \mathsf{thick}(\mathcal{H}_{M})\subseteq {^{\padova}}\mathcal{H}_M,
       \]
       since the right-hand side is already thick. By Proposition \ref{bounded t structure}, it follows that the left-hand side is $\mathcal{T}^{\mathsf{b}}$ and, as a consequence of \cite[Lemma 2.9(v)(b)]{regular}, the right-hand side is $\mathcal{T}^{\mathsf{b}}_p$. 

       (ii)$\implies$(i): Assume now that the inclusion $\mathcal{T}^{\mathsf{b}}\subseteq \mathcal{T}^{\mathsf{b}}_p$ holds for $\mathcal{T}$. From the following equalities
        \[
        \mathcal{T}^{\mathsf{b}}_p={^{\padova}}(\mathcal{T}^{\mathsf{b}})={^{\padova}}(\mathsf{thick}(\mathcal{H}_{M}))={^{\padova}}\mathcal{H}_M, 
        \]
        together with $\mathcal{T}^{\mathsf{b}}=\mathsf{thick}(\mathcal{H}_M)$, it follows that $\mathcal{H}_{M}\subseteq {^{\padova}}\mathcal{H}_M$. We claim that there is an integer $d$ such that $\mathsf{Hom}_{\mathcal{T}}(x,y[>\!d])=0$ for all $x,y\in\mathcal{H}_{M}$. Indeed, if we assume there is no such number, then we can find objects $x_n,y_n\in\mathcal{H}_M$ and a strictly increasing sequence of integers $d_n$ for which 
        \[
        \mathsf{Hom}_{\mathcal{T}}(x_n,y_n[d_n])\neq 0.
        \]
        We may then consider $x=\oplus x_n$, $y=\oplus y_n$, which both belong to $\mathcal{H}_M$, and we have 
        \[
        \mathsf{Hom}_{\mathcal{T}}(x,y[d_n])\neq 0
        \]
        for all $n$, which contradicts $\mathcal{H}_{M}\subseteq {^{\padova}}\mathcal{H}_M$.

        (ii)$\iff$(iii): If $\mathcal{T}^{\mathsf{b}}\subseteq \mathcal{T}^{\mathsf{b}}_p$, then it follows that $(\mathcal{T}^{\mathsf{b}})^{\padova}\supseteq (\mathcal{T}^{\mathsf{b}}_p)^{\padova}$ and by \cite[Lemma 3.2]{regular} we know that $(\mathcal{T}^{\mathsf{b}}_p)^{\padova}=\mathcal{T}^{\mathsf{b}}$. If $\mathcal{T}^{\mathsf{b}}\subseteq \mathcal{T}^{\mathsf{b}}_i$, then it follows that ${^{\padova}}(\mathcal{T}^{\mathsf{b}})\supseteq {^{\padova}}(\mathcal{T}^{\mathsf{b}}_i)$ and by \emph{loc. cit.} we know that ${^{\padova}}(\mathcal{T}^{\mathsf{b}}_i)=\mathcal{T}^{\mathsf{b}}$.

        When $M\in M^{\padova}$, then the inclusion $\mathcal{T}^{\mathsf{b}}_p\subseteq \mathcal{T}^{\mathsf{b}}$ holds (see \cite[Lemma 3.12]{regular}), implying the last claim. 
\end{proof}

The above shows, in particular, that for a connective dg algebra $A$, the finiteness of global dimension is a property intrinsic to $\mathsf{D}(A)$ and has the following direct consequence.

\begin{cor} \label{derived invariance of global dimension}
    Assume a triangle equivalence $\mathsf{D}(A)\simeq \mathsf{D}(B)$ of derived categories of connective dg algebras $A$ and $B$. Then $\gd A<\infty$ if and only if $\gd B<\infty$. 
\end{cor}

\begin{rem} \label{relation to other work}
    The second part of Theorem \ref{global dim vs relative global dim} (which applies to connective dg algebras with finitely many nonzero cohomologies) was obtained independently in \cite{gldim for approximable}, where a notion of global dimension is introduced for a triangulated category $\mathcal{S}$ with a single compact generator $G$. By \cite[Lemma 5.2(1)]{gldim for approximable}, this global dimension is the same as the global dimension of $\mathcal{S}$ relative to the heart of the t-structure induced by $G$ (see \cite[Lemma 2.7]{gldim for approximable}), in the sense of Definition \ref{main definition}. Moreover, it follows by combining \cite[Lemma 5.2(2)]{gldim for approximable} together with \cite[Proposition 3.4]{gldim for approximable} (see also \cite[Remark 5.6]{gldim for approximable}) that when $\mathcal{S}$ is weakly approximable \cite{neeman7} (which is always satisfied when $G$ is silting) and $G$ is bounded (i.e.\ $G\in G^{\padova}$), then the global dimension of $\mathcal{S}$ relative to $G$ is finite if and only if $\mathcal{S}^{\mathsf{b}}=\mathcal{S}^{\mathsf{b}}_p$.
\end{rem}

\begin{lem} \label{lemma}
    Let $(\mathcal{U},\mathcal{V})$ be a t-structure in a triangulated category $\mathcal{T}$ with heart $\mathcal{H}$. Assume the existence of an object $s$ in $\mathcal{H}$ such that every object in the heart admits a finite filtration in $\mathcal{H}$ with quotients in the subcategory $\Add s$. Then $\gd{_{\mathcal{H}}}\mathcal{T}=\pd{_{\mathcal{H}}}s$.
\end{lem}
\begin{proof}
By definition we have that $\pd_{\mathcal{H}}s\leq \gd_{\mathcal{H}}\mathcal{T}$. For the opposite inequality, if $\pd_{\mathcal{H}}s$ were infinite, then there is nothing to prove. We assume therefore that $\pd_{\mathcal{H}}s=n<\infty$ and let $x,y$ be objects in $\mathcal{H}$. By assumption, $x$ admits a filtration with quotients being coproducts of $s$. We will show by induction on the length of this filtration that $\mathsf{Hom}_{\mathcal{T}}(x,y[>\!n])=0$. If $x\in\mathsf{Add}(s)$ (i.e.\ the length of the filtration is zero), then the assertion is clear. If $x$ admits a filtration of length $l$, we may find a triangle 
\[
x_1\rightarrow x\rightarrow s_1\rightarrow  x_1[1]
\]
where $s_1\in\mathsf{Add}(s)$ and $x_1$ admits a filtration of length $l-1$ by objects of $\mathcal{H}$ with quotients coproducts of $s$ (the way we obtain this triangle is by considering first the short exact sequence $0\rightarrow x_1\rightarrow x\rightarrow s_1\rightarrow 0$ in $\mathcal{H}$ and then lifting it to a triangle in $\mathcal{T}$ via the embedding $\mathcal{H}\rightarrow \mathcal{T}$). We have $\pd_{\mathcal{H}}s_1\leq \pd_{\mathcal{H}}s$ and $\pd_{\mathcal{H}}x_1\leq \pd_{\mathcal{H}}s$ by induction, therefore the result follows from Lemma \ref{bound of projective dimensions in a triangle}. 
\end{proof}

\begin{lem} \label{lemma'}
    Assume the setup of Lemma \ref{lemma} and consider an object $x$ that lies in the heart $\mathcal{H}$. If $\mathsf{Hom}_{\mathcal{T}}(x,(-)[j])\colon \mathcal{H}\rightarrow \mathsf{Ab}$ commutes with coproducts for every 
    $j\in\mathbb{Z}$, then 
    \[
    \pd_{\mathcal{H}}x=\mathsf{min}\{n\geq 0: \mathsf{Hom}_{\mathcal{T}}(x,s[>\!n])=0\}.
    \] 
\end{lem}
\begin{proof}
    One inequality is clear. For the converse, simply observe that if $\mathsf{Hom}_{\mathcal{T}}(x,s[>\!n])$ vanishes, then $\mathsf{Hom}_{\mathcal{T}}(x,y[>\!n])=0$ for every object $y$ that belongs to $\Add s$, by using the assumption on $x$. The rest of the proof follows inductively, as in Lemma \ref{lemma}. 
\end{proof}

We will also need the following variant.

\begin{lem} \label{lemma''}
    Let $(\mathcal{U},\mathcal{V})$ be a t-structure in a triangulated category $\mathcal{D}$ with heart $\mathcal{H}$. Assume the existence of an object $s$ in $\mathcal{H}$ such that every object in the heart admits a finite filtration in $\mathcal{H}$ with quotients in the subcategory $\add s$. Then $\gd{_{\mathcal{H}}}\mathcal{D}=\pd{_{\mathcal{H}}}s=\mathsf{min}\{n\geq 0: \mathsf{Hom}_{\mathcal{D}}(s,s[>\!n])=0\}$.
\end{lem}
\begin{proof}
    This is similar to the proofs of Lemma \ref{lemma} and Lemma \ref{lemma'} (and in fact easier as we don't have to deal with coproducts).
\end{proof}

We deduce the following description of the global dimension of a dg algebra. For an ordinary algebra this is a classical result, see \cite[Corollary 11]{auslander}.

\begin{prop} \label{global dimension via simples}
    Let $A$ be a connective dg algebra over a field $k$ that is locally finite, i.e.\ $H^n(A)$ is finite dimensional for every $n$. Consider $s=H^0(A)/\mathsf{rad}H^0(A)$. Then it holds that 
    \[
    \gd A=\pd_As=\mathsf{min}\{n\geq 0:\mathsf{Hom}_{\mathsf{D}(A)}(s,s[>n])=0\}=\gd{_{\smod H^0(A)}}\mathsf{D}_{\mathsf{fd}}(A),
    \]
    where $\mathsf{D}_{\mathsf{fd}}(A)=\{x\in\mathsf{D}(A):\oplus_{n\in\mathbb{Z}}H^n(x)\in\smod k\}$.
\end{prop}
\begin{proof}
    The t-structure $(A^{\perp_{>0}},A^{\perp_{\leq 0}})$ of $\mathsf{D}(A)$ restricts to a t-structure of $\mathsf{D}_{\mathsf{fd}}(A)$ with heart $\smod H^0(A)$, see \cite[Proposition 2.1(b)]{kalck_yang}. As a consequence, since $H^0(A)\in\smod k$, it follows that  $s=H^0(A)/\mathsf{rad}H^0(A)$ satisfies the assumptions of Lemma \ref{lemma''} (because every finitely generated module has a composition series) and so the last equality in the claim holds. For the same reason as before, every $H^0(A)$-module admits a radical series, i.e.\ a filtration with quotients being coproducts of $H^0(A)/\mathsf{rad}H^0(A)$; therefore Lemma \ref{lemma} applies and shows the first equality. For the remaining assertions, we first observe that, by definition, the inequality $\gd_{\smod H^0(A)}\mathsf{D}_{\mathsf{fd}}(A)\leq \gd A$ holds. It therefore remains to show the converse, for which we may assume that $\gd_{\smod H^0(A)}\mathsf{D}_{\mathsf{fd}}(A)$ is finite. It then follows from \cite[Proposition 3.4]{tomonaga} that $\mathsf{D}_{\mathsf{fd}}(A)\subseteq \mathsf{per}(A)$ and in particular $s$ is a compact object of $\mathsf{D}(A)$. As a consequence, Lemma \ref{lemma'} applies to show the middle claimed equality. 
\end{proof}

\begin{rem} \label{remark}
    Tomonaga \cite{tomonaga} defines the global dimension of a connective locally finite dg algebra by 
    \[
    \gd{_{\smod H^0(A)}}\mathsf{D}_{\mathsf{fd}}(A).
    \]
    The above shows that the two notions of global dimension agree. As already mentioned in the previous proof, the latter is finite if and only if $\mathsf{D}_{\mathsf{fd}}(A)\subseteq \mathsf{per}(A)$, which is equivalent to $H^0(A)/\mathsf{rad}H^0(A)\in\mathsf{per}(A)$.
\end{rem}

\begin{rem} \label{koszul}
    Following \cite{fushimi,keller deriving}, under the setup of Proposition \ref{global dimension via simples}, the dg algebra $A^!=\mathsf{REnd}_A(s)$ is the \emph{Koszul dual} of $A$. Proposition \ref{global dimension via simples} therefore tells us that the global dimension of $A$ is equal to the top nonzero degree of the cohomology of $A^!$ (note that this is a coconnective dg algebra, i.e.\ $H^{<0}(A^!)=0$).
\end{rem}

The following characterisation of finite projective dimension is also useful. 
\begin{prop} \label{characterisation of projective dimension}
    Let $\mathcal{T}$ be a compactly generated triangulated category and $M$ a compact silting object of $\mathcal{T}$ that belongs to $\mathcal{T}^{\mathsf{b}}$. The following are equivalent for an object $x$ that lies in $ \mathcal{T}^{\mathsf{b}}\cap M^{\perp_{>0}}$. 
    \begin{itemize}
    \item[\textnormal{(i)}] $x\in \Add M\ast\Add M[1]\ast \cdots\ast\Add M[d]$.
    \item[\textnormal{(ii)}] $\pd{_{\mathcal{H}_M}}x\leq d $.
    \end{itemize}
\end{prop}
\begin{proof}
    It is clear that (i) implies (ii). To show the converse, we observe that condition (ii) holds if and only if $\mathsf{Hom}_{\mathcal{T}}(x,y)=0$ for all $y\in\mathcal{T}^{\mathsf{b}}\cap M^{\perp_{\geq -d}}$ since, by Proposition \ref{bounded t structure}, $(M^{\perp_{>0}},M^{\perp_{\leq 0}})$ restricts to a bounded t-structure in  $\mathcal{T}^{\mathsf{b}}$. We will prove the remaining claim with induction on $d$. We consider the coproduct $M_0=\oplus_{f\in\mathsf{Hom}(M,x)}M$ and the canonical morphism $M_0\rightarrow x$ (this is a right $(\Add M)$-approximation of $x$) and complete it to a triangle 
    \begin{equation} \label{triangle}
        x'\rightarrow M_0\xrightarrow{a}x\rightarrow x'[1].
    \end{equation}
    Applying $\mathsf{Hom}_{\mathcal{T}}(M,-)$ to the above yields a long exact sequence 
    \[
    \cdots \rightarrow  \mathsf{Hom}_{\mathcal{T}}(M,M_0)\xrightarrow{a_*} \mathsf{Hom}_{\mathcal{T}}(M,x)\rightarrow \mathsf{Hom}_{\mathcal{T}}(M,x'[1])\rightarrow \mathsf{Hom}_{\mathcal{T}}(M,M_0[1])\rightarrow \cdots
    \]
    Since $M$ is a compact silting object, it follows that $\mathsf{Hom}_{\mathcal{T}}(M,M_0[>0])= 0$. By the assumption on $x$, we know that $\mathsf{Hom}_{\mathcal{T}}(M,x[>0])= 0$. As a consequence, we see that $\mathsf{Hom}_{\mathcal{T}}(M,x'[\geq 2])=0$. Moreover, by the construction of $M_0$ it is directly verified that $a_*$ is surjective, implying that $\mathsf{Hom}_{\mathcal{T}}(M,x'[1])=0$ too. All in all, we know that $\mathsf{Hom}_{\mathcal{T}}(M,x'[1][\geq\!0])=0$. If $d=0$, then $x'[1]\in\mathcal{T}^{\mathsf{b}}\cap M^{\perp_{\geq 0}}$ and therefore $\mathsf{Hom}_{\mathcal{T}}(x,x'[1])=0$. In particular the triangle (\ref{triangle}) splits and so $x$ belongs to $\Add M$. For $d>0$, it follows as a consequence of the triangle (\ref{triangle}) that $\mathsf{Hom}_{\mathcal{T}}(x',y[>\!d-1])$ vanishes for any object $y$ that lies in the heart $\mathcal{H}_M$. The latter means that $\pd_{\mathcal{H}_M}x'\leq d-1$ which implies -- by the induction hypothesis -- that $x'$ lies in $\Add M\ast \cdots \ast\Add M[d-1]$. Therefore, it follows from (\ref{triangle}) that $x$ belongs to $\Add M\ast \Add M[1]\ast \cdots \ast \Add M[d]$.
\end{proof}

\begin{cor} \label{characterisation of projective dimension for dg alg}
    Let $A$ be a connective dg algebra with bounded cohomology. For any object $x$ that belongs to the intersection $\mathsf{D}^{\mathsf{b}}(A)\cap \mathsf{D}^{\leq 0}(A)$, the following are equivalent. 
    \begin{itemize}
        \item[\textnormal{(i)}] $x\in \Add A\ast \Add A[1]\ast\cdots \ast \Add A[d]$. 
        \item[\textnormal{(ii)}] $\pd_Ax\leq d$.
    \end{itemize}
\end{cor}

\begin{rem} \label{gldim vs minamoto gldim}
    Following Minamoto \cite{minamoto}, for a connective dg algebra $A$, the projective dimension of an $A$-module $x$ that is contained in the heart of the standard t-structure, is equal to $d$ if the functor $F=\mathsf{RHom}_A(x,-)$ maps $\Mod H^0(A)$ to $\mathsf{D}^{[0,d]}(k)$ (complexes with cohomology concentrated in degrees 0 to $d$) and there is $y\in\Mod H^0(A)$ such that $H^dF(y)\neq 0$. It is clear that this coincides with the notion of projective dimension of Definition \ref{main definition} (beware that this is only for objects in the heart; for general modules, \cite{minamoto} follows a different convention), thus the global dimension of \cite[Theorem 5.1]{minamoto} and the one defined in this paper coincide. We also refer to \cite[Theorem 2.22]{minamoto} for a result related to Corollary \ref{characterisation of projective dimension for dg alg} and to \cite[Proposition 5.3]{minamoto}, which we already observed in Corollary \ref{derived invariance of global dimension}, as a direct consequence of intrinsicness. 
\end{rem}

For a $k$-linear compactly generated triangulated category $\mathcal{T}$, we can consider the subcategory of \emph{bounded finite objects}, denoted by $\mathcal{T}^{\mathsf{b}}_c$, which contains the objects $y$ in $\mathcal{T}$ such that 
\[
\oplus_{n\in\mathbb{Z}}\mathsf{Hom}_{\mathcal{T}}(x,y[n])\in\smod k
\]
for all $x\in\mathcal{T}^{\mathsf{c}}$. This notation is also used by Neeman \cite{neeman7} to denote a different (in principle) subcategory, although they sometimes coincide; see \cite[Theorem 1.4]{neeman7}.  Following \cite{regular}, we say that $\mathcal{T}$ is \emph{regular} if $\mathcal{T}^{\mathsf{b}}_c=\mathcal{T}^{\mathsf{c}}$. We will briefly discuss how this is related to the existence of a Serre functor on $\mathcal{T}^{\mathsf{b}}_c$. For that, we recall that a \emph{Serre functor} on a $k$-linear triangulated category $\mathcal{D}$ ($k$ now is assumed to be a field) is a triangle equivalence $\mathrm{S}\colon \mathcal{D}\rightarrow \mathcal{D}$ such that 
\[
\mathsf{Hom}_{\mathcal{D}}(x,y)\cong \mathsf{Hom}_{\mathcal{D}}(y,\mathrm{S}(x))^{\vee}
\]
naturally for every $x,y\in\mathcal{D}$, where $(-)^{\vee}$ stands for the $k$-dual. We have the following general observation, which implies, with Proposition \ref{global dimension via simples}, that if there exists a Serre functor $\mathrm{S}\colon\mathsf{D}_{\mathsf{fd}}(A)\rightarrow \mathsf{D}_{\mathsf{fd}}(A)$, where $A$ is a connective locally finite dg algebra over a field $k$, then the global dimension of $A$ is finite; this was first obtained by Happel \cite{happel serre} in the context of the derived category of a finite dimensional algebra. 

\begin{lem} \label{serre}
    Let $\mathcal{T}$ be a compactly generated triangulated category and $M$ a compact silting object. If $\mathcal{T}^{\mathsf{b}}_c$ admits a Serre functor, then for any object $s\in \mathcal{H}_M\cap \mathcal{T}^{\mathsf{b}}_c$, it holds that $\mathsf{Hom}_{\mathcal{T}}(s,s[n])= 0$ for $n\gg 0$. 
\end{lem}
\begin{proof}
    We already know (see for instance the proof of Theorem \ref{global dim vs relative global dim}) that $\mathcal{H}_M\subseteq{^{\perp_{\ll}}}\mathcal{H}_M$. Moreover, we know that $\mathcal{T}^{\mathsf{b}}=\mathsf{thick}(\mathcal{H}_M)$ and therefore, it follows from \cite[Lemma 2.9(v)(b)]{regular} that $\mathcal{H}_M\subseteq {^{\perp_{\ll}}}(\mathcal{T}^{\mathsf{b}})$. Since there is a Serre functor $\mathrm{S}\colon \mathcal{T}^{\mathsf{b}}_c\rightarrow \mathcal{T}^{\mathsf{b}}_c$, it follows that for every $n\in\mathbb{Z}$, 
    \[
    \mathsf{Hom}_{\mathcal{T}}(s,s[n])\cong \mathsf{Hom}_{\mathcal{T}}(s[n],\mathrm{S}(s))^{\vee}\cong \mathsf{Hom}_{\mathcal{T}}(s,\mathrm{S}(s)[-n])^{\vee}.
    \]
    As a consequence of the above, since $\mathrm{S}(s)$ lies in $\mathcal{T}^{\mathsf{b}}_c\subseteq \mathcal{T}^{\mathsf{b}}$, it follows that $\mathsf{Hom}_{\mathcal{T}}(s,s[n])$ vanishes for $n$ large enough.
\end{proof}

We can now summarise equivalent characterisations of finite global dimension for a proper (meaning that $\oplus_{n\in\mathbb{Z}}H^n(A)\in\smod k$ - equivalently $\mathsf{per}(A)\subseteq \mathsf{D}_{\mathsf{fd}}(A)$) and connective dg algebra $A$ over a field $k$.
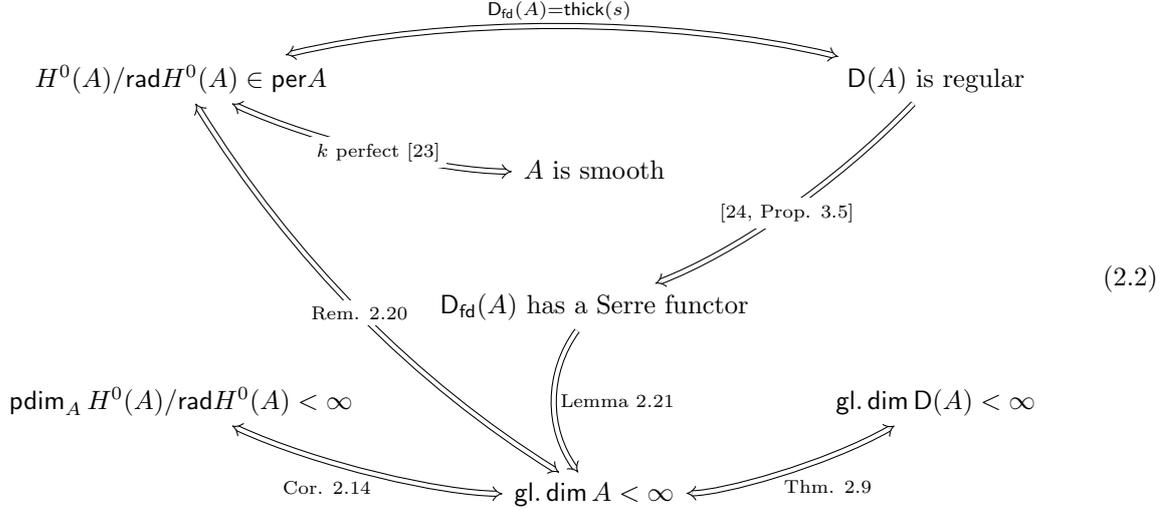
\begin{equation} 
\label{implications}
\begin{tikzcd}
	{H^{0}(A)/\mathsf{rad}H^0(A)\in\mathsf{per}A} && {\mathsf{D}(A)\text{ is regular}} \\
	& {A\text{ is smooth}} \\
	\\
	& {\mathsf{D}_{\mathsf{fd}}(A)\text{ has a Serre functor}} \\
	{\pd_AH^0(A)/\mathsf{rad}H^0(A)<\infty} && {\mathsf{D}(A) \text{ has finite gldim}} \\
	& {\gd A<\infty}
	\arrow["{{{{{\mathsf{D}_{\mathsf{fd}}(A)=\mathsf{thick}(s)}}}}}", curve={height=-24pt}, Rightarrow, 2tail reversed, from=1-1, to=1-3]
	\arrow["{{{{{k \text{ perfect \cite{Raedschelders_Stevenson}}}}}}}"{description}, curve={height=12pt}, Rightarrow, 2tail reversed, from=1-1, to=2-2]
	\arrow["{{{{{\text{Rem. }\ref{remark}}}}}}"{description}, curve={height=18pt}, Rightarrow, 2tail reversed, from=1-1, to=6-2]
	\arrow["{{{{\text{\cite[Prop. 3.5]{tomonaga}}}}}}"{description}, curve={height=-12pt}, Rightarrow, from=1-3, to=4-2]
	\arrow["{{{{{\text{Lemma }\ref{serre}}}}}}", curve={height=18pt}, Rightarrow, from=4-2, to=6-2]
	\arrow["{{{{{{{\text{Prop. }\ref{global dimension via simples}}}}}}}}"', curve={height=12pt}, Rightarrow, 2tail reversed, from=5-1, to=6-2]
	\arrow["{{{{{{{\text{Thm. }\ref{global dim vs relative global dim}}}}}}}}", curve={height=-12pt}, Rightarrow, 2tail reversed, from=5-3, to=6-2]
\end{tikzcd}
\end{equation}

Smoothness for $A$ and finite global dimension of $\mathsf{D}(A)$ were also compared in \cite[Corollary 5.10]{regular}. 

\begin{rem} \label{smoothness implies finite global dimension}
    It is known that for \emph{any} dg algebra $A$ over a field, smoothness implies the inclusion $\mathsf{D}_{\mathsf{fd}}(A)\subseteq \mathsf{per}(A)$ of subcategories of $\mathsf{D}(A)$, see the proof of \cite[Theorem 4.1]{keller} or \cite[Lemma 3.8]{kuznetsov_shinder}.
\end{rem}

\begin{rem}
    As expected, the global dimension of a connective dg algebra $A$ is, in principle, not related to the global dimension of $H^0(A)$. We can consider the example in \cite[Example 7.2]{Raedschelders_Stevenson}; $A=k\langle x,y\rangle$ with $|x|=0$, $|y|=-1$ and $d(y)=x^2$. Then $H^0(A)=k[x]/(x^2)$ has infinite global dimension but $A$ does not, as it is smooth, so $\mathsf{D}_{\mathsf{fd}}(A)\subseteq \mathsf{per}(A)$ from Remark \ref{smoothness implies finite global dimension} and, since $A$ is locally finite the latter implies that $\gd A<\infty$.
\end{rem}

\section{The singularity category of a dg algebra} \label{singularity} Throughout this section we work over a base field $k$. We recall that for a small dg category $\mathcal{A}$ over $k$, the subcategory of $\mathsf{D}(\mathcal{A})$ that consists of finite dimensional modules is defined as 
\[
\mathsf{D}_{\mathsf{fd}}(\mathcal{A})=\{x\in\mathsf{D}(\mathcal{A}):\oplus_{n\in\mathbb{Z}}\mathsf{Hom}_{\mathsf{D}(\mathcal{A})}(h_a,x[n])\in\smod k, \forall a\in\mathcal{A}\},
\]
where $h_a$ stands for the module represented by $a\in\mathcal{A}$ (namely $\mathcal{A}(-,a)$).
The following fact is standard.
\begin{lem} \label{bounded finite objects}
    For any small dg category $\mathcal{A}$ and for $\mathcal{D}=\mathsf{D}(\mathcal{A})$, we have $\mathcal{D}^{\mathsf{b}}_c=\mathsf{D}_{\mathsf{fd}}(\mathcal{A})$.
\end{lem}
\begin{proof}
    By definition, an object $y\in\mathcal{D}$ belongs to $\mathcal{D}^{\mathsf{b}}_c$ if and only if for every $x\in\mathsf{per}(\mathcal{A})$, 
    \[
    \oplus_{n\in\mathbb{Z}}\mathsf{Hom}_{\mathcal{D}}(x,y[n])\in \smod k. 
    \]
    Since $\mathsf{per}(\mathcal{A})=\mathsf{thick}(h_a, a\in\mathcal{A})$, it follows by \cite[Lemma 5.2]{regular} that $y\in\mathcal{D}^{\mathsf{b}}_c$ if and only if 
    \[
    \oplus_{n\in\mathbb{Z}}\mathsf{Hom}_{\mathcal{D}}(h_a,y[n])\in\smod k
    \]
    for all $a\in\mathcal{A}$. The latter translates to $\oplus_{n\in\mathbb{Z}} H^n(y(a))\in\smod k$ or in other words $y(a)\in\mathsf{D}^{\mathsf{b}}(\smod k)$ for all $a\in\mathcal{A}$, which means precisely that $y$ lies in $\mathsf{D}_{\mathsf{fd}}(\mathcal{A})$.
\end{proof}

We recall a concept from \cite{kuznetsov_shinder}, for which we fix from now on a small, idempotent complete dg enhanced triangulated category $\mathcal{A}$. The latter is \emph{reflexive} if the ``coevaluation'' functor $\mathsf{coev}_{\mathcal{A}}\colon\mathcal{A}\rightarrow \mathsf{D}_{\mathsf{fd}}(\mathsf{D}_{\mathsf{fd}}(\mathcal{A}))$ given by $a\mapsto \mathsf{coev}_{\mathcal{A},a}$, where $\mathsf{coev}_{\mathcal{A},a}(x)=x(a)^{\vee}$ and $(-)^{\vee}$ stands for the $k$-dual, is an equivalence (see \cite{kuznetsov_shinder} for details). We treat $\mathsf{D}_{\mathsf{fd}}(\mathcal{A})$ as a dg category and thus we can consider its derived category $\mathsf{D}(\mathsf{D}_{\mathsf{fd}}(\mathcal{A}))$. When $\mathcal{A}$ is additionally proper, then the inclusion $i\colon \mathcal{A}\rightarrow \mathsf{D}_{\mathsf{fd}}(\mathcal{A})$ gives rise to an adjunction
\begin{equation} \label{adjoint triple}
    \begin{tikzcd}
\mathsf{D}(\mathsf{D}_{\mathsf{fd}}(\mathcal{A})) \arrow[rr, "i_*"] &  & \mathsf{D}(\mathcal{A}) \arrow[ll, "i^*"', bend right] \arrow[ll, "i_!", bend left]
\end{tikzcd}
\end{equation}
of triangle functors. For a $k$-linear compactly generated triangulated category $\mathcal{T}$, the \emph{Brown--Comenetz} dual of a compact object $t$ is an object $t^*$ such that 
\[
\mathsf{Hom}_{\mathcal{T}}(-,t^*)\simeq \mathsf{Hom}_k(\mathsf{Hom}_{\mathcal{T}}(t,-),k).
\] 
It is a consequence of Brown representability \cite{neeman duality} that the above object exists. The following result is a version of \cite[Proposition 5.3(iii)]{regular} for reflexive dg categories. 

\begin{prop} \label{proposition appendix}
    Under the above setup, the equality $\mathcal{T}^{\mathsf{b}}_c=i_!(\mathcal{A}^*)$ holds, where $\mathcal{T}=\mathsf{D}(\mathsf{D}_{\mathsf{fd}}(\mathcal{A}))$ and $\mathcal{A}^*$ stands for the subcategory of $\mathsf{D}(\mathcal{A})$ that consists of the Brown--Comenetz duals of (the objects represented by) $\mathcal{A}$.  
\end{prop}
\begin{proof}
    We know from Lemma \ref{bounded finite objects} that $\mathcal{T}^{\mathsf{b}}_c=\mathsf{D}_{\mathsf{fd}}(\mathsf{D}_{\mathsf{fd}}(\mathcal{A}))$. Since $\mathcal{A}$ is reflexive, it follows that every object in the latter category is of the form $\mathsf{coev}_{\mathcal{A},a}$. Moreover, it holds that $i_*\mathsf{coev}_{\mathcal{A},a}\cong \mathcal{A}(a,-)^{\vee}\cong h_a^*$, by definition. We have the following natural isomorphisms for every $x\in\mathsf{D}_{\mathsf{fd}}(\mathcal{A})$
    \begin{align*}
        \mathsf{Hom}_{\mathsf{D}(\mathsf{D}_{\mathsf{fd}}(\mathcal{A}))}(h_x,i_!(h_a^*)) & \cong \mathsf{Hom}_{\mathsf{D}(\mathcal{A})}(i_*(h_x),h_a^*) \\
        & \cong \mathsf{Hom}_{\mathsf{D}(\mathcal{A})}(h_a,i_*(h_x))^{\vee} \\ 
        & \cong i_*(h_x)(a)^{\vee}\cong x(a)^{\vee}
    \end{align*}
    which are induced by the composite 
    \[
    \mathsf{coev}_{\mathcal{A},a}\rightarrow i_!i_*\mathsf{coev}_{\mathcal{A},a}\cong i_!h_a^*,
    \]
    where the left morphism is the unit of the adjunction $(i_*,i_!)$. All in all, it follows that $\mathsf{coev}_{\mathcal{A},a}\cong i_!(h_a^*)$ in $\mathsf{D}(\mathsf{D}_{\mathsf{fd}}(\mathcal{A}))$, which completes the proof.
\end{proof}

\begin{prop} \label{result of appendix}
    Under the above setup consider $\mathcal{T}_{\mathsf{ac}}= \mathsf{D}(\mathsf{D}_{\mathsf{fd}}(\mathcal{A}))/\mathsf{D}(\mathcal{A})$. Then we have $(\mathcal{T}_{\mathsf{ac}})^{\mathsf{b}}_c=0$. 
\end{prop}
\begin{proof}
    There is a recollement of triangulated categories (see for instance \cite{chen chen})
\[\begin{tikzcd}
\mathcal{T}_{\mathsf{ac}} \arrow[rr, "\mathsf{I}"] &  & \mathsf{D}(\mathsf{D}_{\mathsf{fd}}(\mathcal{A})) \arrow[rr, "i_*"] \arrow[ll, bend left] \arrow[ll, "\mathsf{I}_{\lambda}"', bend right] &  & \mathsf{D}(\mathcal{A}) \arrow[ll, "i^*"', bend right] \arrow[ll, "i_!", bend left]
\end{tikzcd}\]
and we know from \cite[Theorem 2.1]{neeman5} that $\mathcal{T}_{\mathsf{ac}}^{\mathsf{c}}=\mathsf{thick}(\mathsf{I}_{\lambda}(\mathsf{D}_{\mathsf{fd}}(\mathcal{A})))$. As a consequence of the latter and the adjunction $(\mathsf{I}_{\lambda},\mathsf{I})$, it follows that an object $x$ belongs to $(\mathcal{T}_{\mathsf{ac}})^{\mathsf{b}}_c$ if and only if $\mathsf{I}(x)$ lies in $\mathcal{T}^{\mathsf{b}}_c=i_!(\mathcal{A}^*)$, thus $x\cong 0$. 
\end{proof}

The following consequence, which is of independent interest, was pointed out to us by an anonymous referee; see \cite{drinfeld} for the notion of a dg quotient.

\begin{cor} \label{not reflexive}
    Let $\mathcal{A}$ be a reflexive and proper dg category. If the dg quotient $\mathsf{D}_{\mathsf{fd}}(\mathcal{A})/\mathcal{A}$ is non-trivial, then it is not reflexive. 
\end{cor}
\begin{proof}
    We have that $\mathsf{D}(\mathsf{D}_{\mathsf{fd}}(\mathcal{A})/\mathcal{A})\simeq \mathsf{D}(\mathsf{D}_{\mathsf{fd}}(\mathcal{A}))/\mathsf{D}(\mathcal{A})$ and thus it follows by Lemma \ref{bounded finite objects} and Proposition~\ref{result of appendix} that $\mathsf{D}_{\mathsf{fd}}(\mathsf{D}_{\mathsf{fd}}(\mathcal{A})/\mathcal{A})=0$, hence $\mathsf{D}_{\mathsf{fd}}(\mathcal{A})/\mathcal{A}$ cannot be reflexive. 
\end{proof}

\begin{cor} \label{regularity for big singularity}
    Let $A$ be a proper connective dg algebra. Then $\gd A<\infty$ if and only if the triangulated category $\mathcal{T}_{\mathsf{ac}}=\mathsf{D}(\mathsf{D}_{\mathsf{fd}}(A))/\mathsf{D}(A)$ is regular.
\end{cor}
\begin{proof}
    We know from \cite[Proposition 6.9]{kuznetsov_shinder} and \cite[Proposition 5.4]{goodbody} that $\mathcal{A}=\mathsf{per}(A)$ is a reflexive dg category and since $\mathsf{D}(\mathsf{D}_{\mathsf{fd}}(A))/\mathsf{D}(A)\simeq \mathsf{D}(\mathsf{D}_{\mathsf{fd}}(\mathcal{A}))/\mathsf{D}(\mathcal{A})$, it follows by Proposition \ref{result of appendix} that $(\mathcal{T}_{\mathsf{ac}})^{\mathsf{b}}_c=0$. Consequently, $\mathcal{T}_{\mathsf{ac}}$ is regular if and only if $\mathcal{T}_{\mathsf{ac}}^{\mathsf{c}}=0$, which is equivalent to $\mathsf{D}_{\mathsf{fd}}(A)/\mathsf{per}(A)\simeq0$. The latter, since $A$ is assumed to be proper, is equivalent to the finiteness of global dimension for $A$, see (\ref{implications}). 
\end{proof}

\section{Global dimension under dg algebra homomorphisms}
We establish a result regarding the global dimension of dg algebras under ring homomorphisms, following \cite{efimov_orlov}, for which we agree on the following notation: for a morphism $h\colon A\rightarrow B$ of dg algebras, we denote by $(h^*,h_*)$ the induction-restriction adjunction of derived functors, as below
\[\begin{tikzcd}
\mathsf{D}(B) \arrow[rr, "h_*"] &  & \mathsf{D}(A) \arrow[ll, "h^*=-\otimes_A^{\mathbb{L}}B"', bend right],
\end{tikzcd}\]
and we recall for later use that the functor $h_*$ is conservative (i.e.\ $h_*(x)\cong 0$ implies that $x\cong 0$). When $A$ and $B$ are connective, then $h_*$ is t-exact, where in both categories we consider the canonical t-structures from Recollection \ref{recollection}. Furthermore, we denote by $T_{h}$ the cone of $h$ in the category of $A$-bimodules and by $T_{h}^{\otimes k}\in\mathsf{D}(A^{\mathrm{op}}\otimes_kA)$ the derived tensor product $T_{h}\otimes_A^{\mathbb{L}}\cdots\otimes_A^{\mathbb{L}} T_{h}$ of $k$ copies of $T_{h}$. Before we prove our next result we need the following lemmata.

\begin{lem} \label{projective dimension of tensor product}
    Let $A$ be a connective dg algebra with bounded cohomology and $x$ an $A$-module that lies in $ \mathsf{D}^{\mathsf{b}}(A)\cap\mathsf{D}^{\leq 0}(A)$. Then for any $A$-bimodule $y$, the inequality $\pd_Ax\otimes_A^{\mathbb{L}}y\leq \pd_Ax+\pd_Ay$ holds. 
\end{lem}
\begin{proof}
    Obviously the claim holds when $\pd_Ax=\infty$, so we assume that $\pd_Ax=n<\infty$. Since $x$ is assumed to belong in the intersection $ \mathsf{D}^{\mathsf{b}}(A)\cap \mathsf{D}^{\leq 0}(A)$, it follows from Corollary \ref{characterisation of projective dimension for dg alg} that $x$ belongs to $\Add A\ast \Add A[1]\ast \cdots \ast \Add A[n]$. As a consequence, $x\otimes_A^{\mathbb{L}}y$ lies in $\Add y\ast \Add y[1]\ast\cdots\ast \Add y[n]\subseteq \mathsf{D}(A)$. By definition of the projective dimension, we have that $\pd_Ay[m]= \pd_Ay+m$ for any integer $m$ and therefore, using Lemma \ref{bound of projective dimensions in a triangle}, it follows that $\pd_A x\otimes_A^{\mathbb{L}}y\leq \pd_Ay+n$. 
\end{proof}

\begin{lem} \label{octahedral}
    Let $g\colon x\rightarrow y$ and $f\colon y\rightarrow x$ be morphisms in a triangulated category $\mathcal{T}$ with $fg\simeq \mathsf{Id}_x$. Then $\mathsf{cone}(f)\cong \mathsf{cone}(g)[1]$. 
\end{lem}
\begin{proof}
    As a consequence of the octahedral axiom, there exists a triangle 
    \[
    \mathsf{cone}(g)\rightarrow \mathsf{cone}(f g)\rightarrow \mathsf{cone}(f)\rightarrow \mathsf{cone}(g)[1]
    \]
    and since $fg\simeq \mathsf{Id}_x$, the claim follows.
\end{proof}

The following observation is used in \cite{efimov_orlov} -- we give a proof for convenience. 
\begin{lem} \label{isomorphism}
    Let $h\colon A\rightarrow B$ be a morphism of dg algebras. For a $B$-module $x$, set $x_0=x$ and define inductively $x_n$ to be the cone of $h^*h_*(x_{n-1})\rightarrow x_{n-1}$ in $\mathsf{D}(B)$. Then $h_*(x_n)\cong h_*(x)\otimes_A^{\mathbb{L}}T_h^{\otimes n}[n]$.
\end{lem}
\begin{proof}
    We first apply $h_*(x)\otimes^{\mathbb{L}}_A-$ to the triangle $A\xrightarrow{h}B\rightarrow T_h\rightarrow A[1]$ in $\mathsf{D}(A^{\mathrm{op}}\otimes_kA)$ to get a triangle 
    \[
    h_*(x)\xrightarrow{1\otimes h} h_*(x)\otimes^{\mathbb{L}}_AB\rightarrow h_*(x)\otimes^{\mathbb{L}}_AT_h\rightarrow h_*(x)[1]
    \]
    in $\mathsf{D}(A)$. We have that $h_*(\epsilon_x)\circ (1\otimes h)\simeq \mathsf{Id}_{h_*(x)}$, where $\epsilon$ denotes the counit of the adjunction $(h^*,h_*)$ and therefore, we derive from Lemma \ref{octahedral} that $h_*(x_1)\cong h_*(x)\otimes^{\mathbb{L}}_AT_h[1]$. Inductively we deduce the isomorphisms $h_*(x_n)\cong h_*(x_{n-1})\otimes_A^{\mathbb{L}}T_h[1]\cong (h_*(x)\otimes_A^{\mathbb{L}} T_h^{\otimes n-1}[n-1])\otimes_A^{\mathbb{L}}T_h[1]\cong h_*(x)\otimes^{\mathbb{L}}_AT_h^{\otimes n}[n]$.
\end{proof}

The following is a version of \cite[Theorem 8]{efimov_orlov} for global dimension. 

\begin{thm} \label{dg algebra homomorphism}
    Let $h\colon A\rightarrow B$ be a homomorphism of connective dg algebras with bounded cohomology and assume that $T_h^{\otimes n}=0$ for some $n\geq 1$. If $\pd_AT_h\leq d$, then the inequality
    \[
    \gd B\leq \gd A+ (n-1)(d+1)
    \]
    holds. In particular, if $\gd A$ is finite then $\gd B$ is finite. 
\end{thm}
\begin{proof}
    Let $x$ be a $B$-module, set $x_0=x$ and define inductively $x_n$ to be the cone of the counit $h^*h_*(x_{n-1})\rightarrow x_{n-1}$ in $\mathsf{D}(B)$. For every $k\geq 0$, $m\in\mathbb{Z}$, and any $y\in B^{\perp_{\neq 0}}$ we have 
    \begin{align*}
        \mathsf{Hom}_{\mathsf{D}(B)}(h^*h_*(x_k),y[m]) &\cong \mathsf{Hom}_{\mathsf{D}(A)}(h_*(x_k),h_*(y)[m]) \\ 
        & \cong \mathsf{Hom}_{\mathsf{D}(A)}(h_*(x)\otimes^{\mathbb{L}}_AT_h^{\otimes k}[k],h_*(y)[m]) \ \ \ \ \  \text{Lemma }\ref{isomorphism} \\
        & \cong \mathsf{Hom}_{\mathsf{D}(A)}(h_*(x)\otimes^{\mathbb{L}}_AT_h^{\otimes k},h_*(y)[m-k]).
    \end{align*}
    By using the triangle $A\rightarrow B\rightarrow T_h\rightarrow A[1]$ we see that 
    $\mathsf{Hom}_{\mathsf{D}(A)}(A,T_h[>\!0])=0$ and therefore $T_h$ belongs to $\mathsf{D}^{\mathsf{b}}(A)\cap \mathsf{D}^{\leq 0}(A)$. If we assume that $x$ belongs to $B^{\perp_{\neq 0}}$, then by using Lemma \ref{projective dimension of tensor product} we see inductively that $\pd_A h_*(x)\otimes^{\mathbb{L}}_AT_h^{\otimes k}\leq \pd_A h_*(x)+kd$. In view of the isomorphisms above, it follows that
    \begin{equation} \label{inequality}
        \pd_B h^*h_*(x_k)\leq \pd_A h_*(x)+kd+k.
    \end{equation}
    Moreover, since $T_h^{\otimes n}=0$, it follows from Lemma \ref{isomorphism} that $h_*(x_n)\cong 0$ and as a consequence $x_n\cong 0$. We may now consider the triangles 
    \[
    h^*h_*(x)\rightarrow x\rightarrow x_1\rightarrow  h^*h_*(x)[1], \dots, h^*h_*(x_{n-1})\xrightarrow{\cong} x_{n-1}\rightarrow 0\rightarrow h^*h_*(x_{n-1})[1].
    \]
    By the first triangle together with (\ref{inequality}) and Lemma \ref{bound of projective dimensions in a triangle} it follows that 
    \[
    \pd_{B}x\leq \mathsf{max}\{\pd_Bx_1,\pd_{A}h_*(x)\}. 
    \]
    For the same reason, and by using the triangle $h^*h_*(x_1)\rightarrow x_1\rightarrow x_2\rightarrow h^*h_*(x_1)[1]$, it follows that 
    \[
    \pd_Bx_1\leq \mathsf{max}\{\pd_Bx_2,\pd_Ah_*(x)+d+1\}.
    \]
    We iterate this argument and after finitely many steps (due to the isomorphism $h^*h_*(x_{n-1})\cong x_{n-1}$), we arrive to the inequality $\pd_Bx\leq \pd_Ah_*(x)+(n-1)(d+1)$, which implies the claim. 
\end{proof}

\begin{prop} \label{acyclic quiver}
    Let $A=kQ/I$ be a bound quiver algebra where $Q$ is a finite quiver, $I$ an admissible ideal and $k$ a field. Assume that $kQ/I$ is given a connective dg algebra structure concentrated in degrees $-d$ to $0$. If $Q$ has no oriented cycles, then 
    \[
    \gd kQ/I\leq l(d+1),
    \]
    where $l$ is the maximal length of a path in $Q$. 
\end{prop}
\begin{proof}
    Consider the canonical morphism of dg algebras $h\colon kQ_0\rightarrow kQ/I$, where $kQ_0$ is regarded as a dg algebra in degree zero (this is a homomorphism of dg algebras by the assumption of connectivity for $kQ/I$). Since $h$ is injective, it follows that $T_h$ is isomorphic, in $\mathsf{D}(kQ_0)$, to the cokernel of $h$ and its underlying $kQ_0$-bimodule is the kernel of the epimorphism $kQ/I\rightarrow kQ_0$, say $M$. In particular, we have that $T_h^{\otimes l+1}=0$ (since the same holds for $M$, see for instance \cite[Corollary 7.12]{beligiannis}) and $\pd_{kQ_0}T_h\leq d$ (since $kQ_0$ is semisimple and $T_h$ is $kQ_0$-dg-bimodule in degrees $-d$ to $0$), so the result follows from  Theorem \ref{dg algebra homomorphism}. 
\end{proof}

\begin{cor}
    Let $A=kQ/I$ be a bound quiver algebra over a field $k$ that is given a connective dg algebra structure for which $\gd kQ/I= n$. Then there is a path in $Q$ with length at least $ \lfloor\frac{n}{d+1}\rfloor$, where $d$ is the maximum integer such that $A^{-d}\neq 0$. 
\end{cor}
\begin{proof}
    Assume that every path in $Q$ has length at most $\lfloor\frac{n}{d+1}\rfloor-1$. Then, in particular, $Q$ does not have oriented cycles and therefore it follows by Proposition \ref{acyclic quiver} that 
    \[
    \gd A< \frac{n}{d+1}(d+1)=n,
    \]
    a contradiction. 
\end{proof}

For $d=0$ the above is shown in \cite[Proposition 3.3]{happel_zacharia} and, in simple terms, it says that if $\gd kQ/I$ is large, then the quiver should have paths with large enough length. When $d\neq 0$, this behaviour is no longer expected; consider the quiver $Q=\begin{tikzcd}
	1 & 2
	\arrow["\alpha", from=1-1, to=1-2]
\end{tikzcd}$ and view its path algebra as a dg algebra with $\mathsf{deg}(\alpha)=-d$ and trivial differential. Then, by \cite[Corollary 3.10]{tomonaga}, we have $\gd kQ=  d+1$.

\end{document}